\documentclass{aims}
\usepackage{amsmath}
\usepackage{paralist}
\usepackage{graphics} 
\usepackage{epsfig} 
\usepackage{graphicx}  \usepackage{epstopdf}
\usepackage[colorlinks=true]{hyperref}
\usepackage{xcolor}
\hypersetup{urlcolor=blue, citecolor=red}

\newtheorem{theorem}{Theorem}[section]

\newtheorem{lemma}[theorem]{Lemma}
\newtheorem{proposition}{Proposition}[section]

\theoremstyle{definition}
\newtheorem{definition}[theorem]{Definition}
\newtheorem{remark}{Remark}

\title[Weak solvability of a fluid-like driven system] 
{Weak solvability a fluid-like driven system for active-passive pedestrian dynamics}

\author[T. K. Thoa Thieu and Matteo Colangeli and Adrian Muntean]{}

\subjclass{MSC: 34B60, 34D20, 35Q35, 35K55, 35K65, 76S05, 76S99}
\keywords{Pedestrian flows, Nonlinear coupling, Forchheimer flows, Double nonlinear parabolic equation.}

\email{thikimthoa.thieu@gssi.it}
\email{matteo.colangeli1@univaq.it }
\email{adrian.muntean@kau.se}

\begin{document}
	\maketitle
	
	\centerline{\scshape T. K. Thoa Thieu$^*$}
	\medskip
	{\footnotesize
		\centerline{Department of Mathematics, Gran Sasso Science Institute,  }
		\centerline{Viale Francesco Crispi 7, L'Aquila 67100, Italy}
		\centerline{Department of Mathematics and Computer Science, Karlstad University,}
		\centerline{Universitetsgatan 2, Karlstad, Sweden}
	} 
	
	\medskip
	
	\centerline{\scshape Matteo Colangeli}
	\medskip
	{\footnotesize
		\centerline{ Dipartimento di Ingegneria e Scienze dell'Informazione e Matematica, Universit\`{a} degli Studi dell'Aquila,}
		\centerline{Via Vetoio, L'Aquila 67100, Italy}
	}
	
	\medskip
	
	\centerline{\scshape Adrian Muntean}
	\medskip
	{\footnotesize
		\centerline{Department of Mathematics and Computer Science, Karlstad University,}
		\centerline{Universitetsgatan 2, Karlstad, Sweden}
	}
	\bigskip
	

	\begin{abstract}
		We study the question of  weak solvability for a nonlinear coupled parabolic system that models the evolution of a complex pedestrian flow. The main feature is that the flow is composed  of a mix of densities of active and passive pedestrians that are moving with different velocities. We rely on special energy estimates and on the use a Schauder's fixed point argument to tackle the existence of solutions to our evolution problem. 
		\end{abstract}
	
	
	\section{Introduction}\label{Sec:Setting}
	In this paper, we study the weak solvability of a coupled system of parabolic equations which are meant to describe the motion of a pedestrian flow in a heterogeneous environment. From the crowd dynamics perspective, the standing assumption is that our pedestrian flow is composed  of two distinct populations: an {\em active} population  -- pedestrians are aware of the details of the environment and move rather fast,  and a {\em passive} population -- pedestrians are not aware of the details of the environment and move therefore rather slow.  See also Figure \ref{fig:fig0}, where we make the analogy with flow in a structured porous media, following an idea by Barenblatt and co-authors cf. \cite{Barenblatt1960}. Mathematically, we investigate an evolution system where a Forchheimer-like equation is nonlinearly-coupled to a diffusion-like equation. For more details on the modeling of the situation, we refer the reader to Section  \ref{Sec:intro} and references mentioned there. 
	
	Let a bounded set $\Omega \neq \emptyset$, $\Omega \subseteq \mathbb{R}^2$ a domain such that $\partial \Omega=\Gamma^N \cup \Gamma^R$, $\Gamma^N \cap \Gamma^R= \emptyset$ with $\mathcal{H}(\Gamma^N) \neq \emptyset$ and $\mathcal{H}(\Gamma^R)\neq \emptyset$, where $\mathcal{H}$ denotes the surface measure on $\Gamma^N, \Gamma^R$ and take $S=(0,T)$. Find the pair $(u,v)$, where  $u: S \times \Omega \longrightarrow \mathbb{R}^2$ and $v: S\times\Omega \longrightarrow \mathbb{R}^2$,  satisfying the following model equations 
	\begin{align}\label{main_eq}
	\begin{cases}
	\partial_t (u^\lambda) + \mathrm{div}(-K_1(|\nabla u|)\nabla u) = -b(u-v) \ \text{ in } S \times \Omega,\\
	\partial_t v - K_2\Delta v =b(u-v) \ \text{ in } S \times \Omega, \\
	-K_1(|\nabla u|)\nabla u \cdot {\bf n} = \varphi u^\lambda \ \text{ at } S \times\Gamma^R,\\
	-K_1(|\nabla u|)\nabla u\cdot {\bf n} = 0 \ \text{ at } S\times\Gamma^N ,\\
	-K_2\nabla v \cdot {\bf n} = 0 \ \text{ at }  S\times \partial\Omega, \\
	u(t=0,x) = u_0(x), \ x\in  \bar{\Omega}, \\
	v(t=0,x) = v_0(x), \ x\in\bar{\Omega}.
	\end{cases}
	\end{align}
	Here $K_2 >0$, 
	 while the function $K_1$ is linked to the derivation of a nonlinear version of Darcy's law involving   a polynomial with non-negative coefficients in velocities. This choice is rather non-standard, see e.g. the works  \cite{Hoang2011}, \cite{Aulisa2009}, \cite{Celik2016} and references cited therein for more details in this sense. In addition, $\lambda \in (0,1]$ is a fixed number and $b(\cdot)$ is a sink/source term. The nonlinear structure of $K_1$ is described in Section \ref{Sec:Pre} together with the remaining model parameters entering (\ref{main_eq}) which are not explained here, as well as with the assumptions needed to ensure the existence of solutions to our problem. 
	
	\begin{figure}
		\centering
		\includegraphics[width=0.7\textwidth]{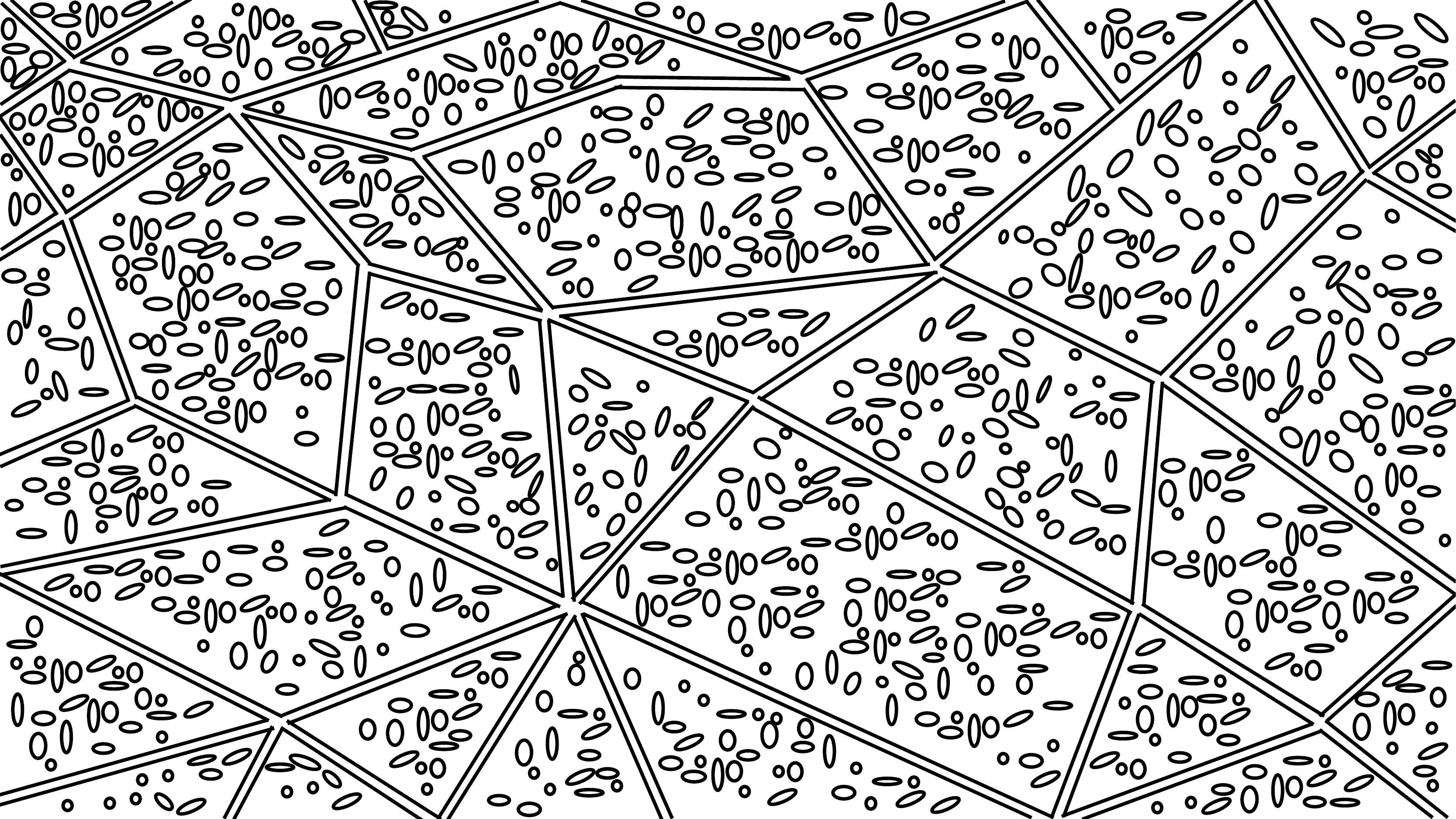}
		\caption{Sketch of a distributed flow through a fissured rock, scenario mimicking Fig.1 from \cite{Barenblatt1960}. The fissured rock consists of pores and permeable blocks, generally speaking blocks are separated from each other by a system of fissures. Through the fissures, the flow is faster compared to the rest of the media.}
		\label{fig:fig0}
	\end{figure}
	
	\section{Some background on the problem (\ref{main_eq})}\label{Sec:intro}
	
	The modeling, analysis and simulation of pedestrian flows offers many  challenging questions to science and technology in general, and to mathematics in particular. Our  interest in this context is to study mixed pedestrian flows where the dynamics of interacting agents stems from two distinct populations: active agents that follow a predetermined velocity field and  passive
	agents that have no preferred direction of motion. 
	
	There are several ways to approach such scenarios. One possible route has been studied in \cite{Richardson2019, Colangeli2019}, where the authors considered a non-linear system SDEs coupled with a linear parabolic equation to describe the escape evacuation dynamics of  active and passive pedestrians moving through smoke (i.e. through regions with reduced visibility). A lattice model is explored in \cite{Cirillo2019} to search for eventual drafting/aerodynamic drag effects.

	
In this paper, we imagine that the active population of pedestrians have velocities similar to a non-Darcy flow,  namely, a generalized Forchheimer flow as for  slightly compressible fluids. Some of the mathematical properties of  this type of flow have been already investigated, for instance, in \cite{Aulisa2009, Hoang2011, Celik2016}, and we are benefitting here of this background.  On the other hand, we consider flow of the passive population as a diffusion process, hence no predetermined flow directions are prefered. The coupling between the two flows is done like in \cite{Barenblatt1960}.
	
	

	From a micro-to-macro perspective, it is worth also noting that a a generalized Forchheimer flow model (i.e. the first partial differential equation in  (\ref{main_eq}))  can be obtained in principle via homogenization techniques (like in  \cite{Lions2011}, e.g.), but it is not clear at this stage how would look like a suitable microscopic model defined at the level of the geometry depicted in Figure \ref{fig:fig0}. 
	
	 This paper is organized as follows. In Section \ref{Sec:Pre}, some preliminaries and assumptions are provided. Then, we prove the special energy estimates in Section \ref{Sec:Energy}. Finally, the existence of solutions to the problem \eqref{main_eq} is established in Section \ref{existence}.
	
	\section{Preliminaries. List of assumptions}\label{Sec:Pre}
	
	We list in this section a couple of preliminary results (mostly inequalities and compactness results) as well as our assumption on data and parameters.
	\begin{lemma} Let $x, y \in [0,\infty)$. Then the following elementary inequalities hold:
		\begin{equation}\label{pre-ine-1}
		(x+y)^p \leq 2^p(x^p+y^p) \text{ for all } p >0, 
		\end{equation}
		\begin{equation}\label{pre-ine-2}
		(x+y)^p \leq x^p+y^p \text{ for all } 0 <p \leq 1,
		\end{equation}
		\begin{equation}\label{pre-ine-3}
		(x+y)^p \leq 2^{p-1}(x^p+y^p) \text{ for all } p \geq 1,
		\end{equation}
		\begin{equation}\label{pre-ine-4}
		x^{\beta} \leq x^{\alpha} + x^{\gamma} \text{ for all } 0 \leq \alpha \leq \beta \leq \gamma,
		\end{equation}
		\begin{equation}\label{pre-ine-5}
		x^\beta \leq 1+x^{\gamma} \text{ for all } 0 \leq \beta \leq \gamma.
		\end{equation}
	\end{lemma}
	\begin{lemma}[\bf A trace inequality]\label{trace-lm}
		Let $\lambda \in (0,1]$, $\delta =1-\lambda$, $a  \in (0,1)$, $a>\delta$, $\alpha \geq 2-\delta$, $\alpha \leq 2$, $\mu_0=\frac{a-\delta}{1-a}$, $\alpha_{\star}=\frac{n(a-\delta)}{2-a}$ and 
		\begin{align}
		\theta = \theta_\alpha := \frac{1}{(1-a)(\alpha/\alpha_{\star}-1)} \in (0,1).
		\end{align}
		Then it exists $C>0$ such that the following estimate holds
		\begin{align}\label{trace-theo}
		\int_{\Gamma^R}|u|^\alpha d\sigma \leq 2\varepsilon \int_{\Omega} |u|^{\alpha+\delta-2}|\nabla u|^{2-a}dx + C\|u\|_{L^{\alpha}(\Omega)}^\alpha \nonumber\\+ C\varepsilon^{-\frac{1}{1-a}}\|u\|_{L^{\alpha}(\Omega)}^{\alpha+\mu_0} + C\varepsilon^{-\mu_2}\|u\|_{L^\alpha(\Omega)}^{\alpha+\mu_1},
		\end{align}
		where
		\begin{align}
		\mu_1&=\mu_{1,\alpha} := \frac{\mu_0(1+\theta(1-a))}{1-\theta},\\
		\mu_2&=\mu_{2,\alpha}:=\frac{1}{1-a}+\frac{\theta(2-a)}{(1-\theta)(1-a)}.
		\end{align}
		\end{lemma}
	For the proof of Lemma $2.2$, see Lemma $2.2$  in \cite{Celik2016}. Here $a = \frac{\alpha_N}{\alpha_N+1} \in (0,1)$, with  $\alpha_N$ cf. (\ref{g}). 
	
	\begin{theorem}[\bf Rellich-Kondrachov Compactness Theorem \cite{Evans1998}]\label{rellich-theo}
		Suppose $\Omega$ is bounded open subset of $\mathbb{R}^d$ and $\partial \Omega$ is $C^1$. If $1 \leq p <d$, then 
		$W^{1,p}(\Omega) \hookrightarrow\hookrightarrow L^q(\Omega)$ for each $1\leq q <p^{\star}$ with $p^{\star} = \frac{pd}{d-p}$.
	\end{theorem}
	\begin{theorem}[\bf Aubin-Lions compactness Theorem \cite{Aubin1963}]\label{Aubin-theo} Let $E_0 \hookrightarrow \hookrightarrow E \hookrightarrow E_1$ be three Banach spaces. Suppose that $E_0$ is compactly embedded in $E$ and $E$ is continuously embedded in $E_1$. Then 
		\begin{align*}
		W=\left\{u\in L^p(S; E_0) \text{ and } u_t \in L^q(S; E_1) \right\}
		\end{align*}
		is compactly embedded into $L^p(S; E)$.
	\end{theorem}
	\begin{theorem}[\bf Schauder's Fixed Point Theorem \cite{Zeidler1986}]\label{Schauder-theo} Let $B$ be a nonempty, closed, bounded, convex subset of a Banach space $X$, and suppose: $\mathcal{T}: B \longrightarrow B$ is a compact operator. Then $\mathcal{T}$ has a fixed point.
		
	\end{theorem}
	In the sequel, we recall some definitions on the constructions on the nonlinear Darcy equation and its monotonicity properties as they have been presented in \cite{Aulisa2009}.
	First of all, we introduce the function $K_1: \mathbb{R}^{+}\longrightarrow \mathbb{R}^{+}$ defined for $\xi \geq 0$ by $K_1(\xi)=\frac{1}{g(s(\xi))}$ being the unique non-negative solution of the equation $sg(s) = \xi$, where $g :\mathbb{R}^{+} \to \mathbb{R}^{+}$ is a polynomial with positive coefficients defined by 
	\begin{align}\label{g}
	g(s) = a_0s^{\alpha_0} + a_1s^{\alpha_1} + \ldots + a_Ns^{\alpha_N} \ \text{ for } s \geq 0,
	\end{align}
	where $\alpha_k \in \mathbb{R}_{+}$ with $k\in \{0,\ldots,N\}$.\\
	The function $g$ be independent of the spatial variables. Thus, we may have
	\begin{align}\label{nonlinearDarcy}
	G(|v|)=g(|v|)|v|=|\nabla p|,
	\end{align}
	where $G(s)=sg(s)$ for $s \geq 0$. From now on we use the following notation for the function $G$ and its inverse $G^{-1}$, namely, $G(s) = sg(s)=\xi$ and $s=G^{-1}(\xi)$. To be successful the analysis below, we impose the following condition, referred to $(G)$.

	\textit{$(G_1)$ } $g \in C([0,\infty)) \cap C^1((0,\infty))$ such that
	$$g(0) > 0 \text{ and } g'(s) \geq 0 \text{ for all } s \geq 0.$$
	
	\textit{$(G_2)$ } It exists $\theta >0 $ with $g \in C([0,\infty)) \cap C^1((0,\infty))$ such that
	\begin{align}
	g(s) \geq \theta sg'(s) \text{ for all } s>0.
	\end{align}

	\subsection{Assumptions}
	We make the following choices on the structure of the nonlinearities.
	\begin{itemize}
		\item[($\text{A}_1$)] The structure of $K_1(\xi)$ has the following properties hold $K_1: [0,\infty) \longrightarrow (0,\frac{1}{a_0}]$ such that $K_1$ is decreasing and
		\begin{align}\label{k1-eqn1}
		\frac{d_1}{(1+\xi)^a} \leq K_1(\xi) \leq \frac{d_2}{(1+\xi)^a};
		\end{align}
		\begin{align}\label{k1-eqn2}
		d_3(\xi^{2-a}-1)\leq K_1(\xi)\xi^2 \leq d_2\xi^{2-a} \text{ for all } \xi \in [0,\infty).
		\end{align}
		In \eqref{k1-eqn1}, $d_1, d_2, d_3$ are strictly positive constants depending on $g(s)$ and $a \in (0,1)$. 
		\item[($\text{A}_2$)] The function $b: \mathbb{R} \longrightarrow \mathbb{R}$ satisfies the following structural condition: it exits $\hat{c}>0$ such that $b(z) \leq \hat{c}|z|^{\sigma}$, with $\sigma \in (0,1)$. 
	\end{itemize}
	Note that the choice of $(\text{A}_1)$ was inspired by Theorem III.10 in \cite{Aulisa2009}, while $(\text{A}_2)$ is a coupling refereed to as Henry term in a mass transfer context (see e.g. \cite{Muntean2010} and \cite{Lind2018} for a related setting). 
	We can now define the following concept of solutions to \eqref{main_eq}.
	\begin{definition}\label{dn-weak}
		Find $$(u,v) \in \left(L^\alpha(S;L^\alpha(\Omega))\cap L^{2-a}(S;W^{1,2-a}(\Omega))\right) \times L^2(S;W^{1,2}(\Omega))$$
		satisfying the identities
		\begin{align}\label{weak1}
		\int_{\Omega}\partial_t(u^{\lambda})\psi dx + \int_{\Omega} K_1(|\nabla u|)\nabla u \nabla \psi dx + \int_{\Gamma^R}\varphi u^{\lambda}\psi d\gamma = -\int_{\Omega}b(u-v)\psi dx
		\end{align}
		and
		\begin{align}\label{weak2}
		\int_{\Omega}\partial_t v\phi dx + \int_{\Omega}K_2 \nabla v \nabla \phi dx = \int_{\Omega}b(u-v)\phi dx
		\end{align}
		for a.e $t \in S$ and for all $\psi \in L^\alpha(\Omega)$, $\phi \in W^{1,2}(\Omega)$.
	\end{definition}
	\subsection{Statement of the main result}
	The main result of this paper is stated in Theorem \ref{existence-theo}. Here we establish the existence of solutions. It turns out that for completing the well-posedness study of our system much more information is needed. We comment on this matter in Remark \ref{discussion}.  
	
	\begin{theorem}\label{existence-theo}
		Assume that $(A_1)$-$(A_2)$ hold. Let $\lambda \in (0,1]$, $\delta =1-\lambda$, $a = \frac{\alpha_N}{\alpha_N+1} \in (0,1)$, $a > \delta$, $\alpha \geq 2-\delta$, $\alpha \leq 2$, $\sigma\leq \frac{\alpha}{2}$, $\sigma \in (0,1)$ and $u_0 \in L^{\alpha}(\Omega)$, $v_0 \in L^2(\Omega)$. Then the problem \eqref{main_eq} has at least a weak solution\\ $(u,v) \in \left(L^\alpha(S;L^\alpha(\Omega))\cap L^{2-a}(S;W^{1,2-a}(\Omega))\right) \times L^2(S;W^{1,2}(\Omega))$ in the sense of Definition \ref{dn-weak}. 
	\end{theorem}
	
\begin{remark}\label{discussion}
	To obtain information concerning the uniqueness of solutions or about and eventual the stability with respect to data and parameters, we conjecture that $\nabla u \in L^\infty(S; L^\infty(\Omega))$ and $\nabla v \in L^\infty(S; L^\infty(\Omega))$.  This way of arguing is in line with Proposition IV.4 in \cite{Aulisa2009}. However, this sort of Bernstein-like estimates  on the solutions in the sense of Definition \ref{dn-weak} are not yet available.	We are currently attempting to prove them using techniques from  \cite{Ladyzenskaja1968}.
\end{remark}
	
	\section{Energy estimates}\label{Sec:Energy}
	In this section, we provide the energy estimates for the solutions in the sense of Definition \ref{dn-weak} to our problem \eqref{main_eq}. This is a crucial step, which in fact allows the Schauder-fixed point Theorem to work in our case.
	
	\subsection{$L^{\alpha}$ - $L^2$ estimates}
	\begin{proposition}\label{prop1}
		Assume that $(A_1)$-$(A_2)$ hold and let $\lambda \in (0,1]$, $\delta =1-\lambda$, $a = \frac{\alpha_N}{\alpha_N+1} \in (0,1)$, $a > \delta$, $\alpha \geq 2-\delta$, $\alpha \leq 2$, $\sigma\leq \frac{\alpha}{2}$, $\sigma \in (0,1)$ and $u_0 \in L^{\alpha}(\Omega)$, $v_0 \in L^2(\Omega)$. Then, for any $t\in S$, the following estimates hold
		\begin{align}\label{main_0}
		\frac{d}{dt}\int_{\Omega}|u|^{\alpha}dx + \int_{\Omega}|\nabla u|^{2-a}|u|^{\alpha + \delta -2}dx \leq C_1+\Big(\frac{3}{2C_2}\hat{c}+\nonumber\\ + \frac{d_3(\alpha - \lambda)}{C_2}\Big)\|u\|_{L^{\alpha}(\Omega)}^{\alpha} + \frac{\hat{c}}{2C_2}\|v\|_{L^2(\Omega)}^2.
		\end{align}
		\begin{align}\label{ine_main1}
		\int_{\Omega}|u|^{\alpha}dx + \int_{\Omega}v^2dx \leq e^{C_3t}\left(1+\|u_0\|_{L^\alpha(\Omega)}^{\alpha} + \|v_0\|_{L^2(\Omega)}^2\right),
		\end{align}
		\begin{align}\label{ine_main2}
		\int_0^T\int_{\Omega}|u|^{\alpha+\delta-2}|\nabla u|^{2-a}dxdt + \int_{0}^{T}\int_{\Omega}|\nabla v|^2dxdt \leq C_5 +\nonumber\\+ C_6\left(\|u_0\|_{L^\alpha(S; L^\alpha(\Omega))}^{\alpha} + \|v_0\|_{L^2(S; L^2(\Omega))}^2 \right),
		\end{align}
		where $C_1:= \frac{d_3(\alpha-\lambda)+\hat{c}}{C_5}|\Omega|$, $C_2:=\min\left\{\frac{\lambda}{\alpha},d_3(\alpha-\lambda)\right\}$, $C_3:=\max\left\{\frac{5}{2}\underline{\tilde{c}}\hat{c}|\Omega|,2\underline{\tilde{c}}\hat{c}\right\}$, $C_4:= \min\left\{\alpha-\lambda,K_2\right\}$,  $C_5:=\frac{5T}{2C_4}\hat{c}|\Omega| + \frac{2T\hat{c}e^{C_3t}}{C_4}$,  and   $C_6:=\frac{2\hat{c}e^{C_3t}}{C_2}$ with $\underline{\tilde{c}}:=\frac{1}{\min\left\{\frac{\lambda}{\alpha},\frac{1}{2}\right\}}$ and $\hat{c}$ as in $(\text{A}_2)$, respectively.
	\end{proposition}
	\begin{proof}
		To prove \eqref{ine_main1}, we proceed as follows. We consider firstly the following sub-problem to which we refer to as $(P_1)$: For a given $v \in L^2(S; L^2(\Omega))$, search for $u \in L^\alpha(S;L^\alpha(\Omega))\cap L^{2-a}(S;W^{1,2-a}(\Omega))$ such that \eqref{sub-prob1} is fulfilled, viz. 
		\begin{align}\label{sub-prob1}
		\begin{cases}
		\partial_t (u^\lambda) + \mathrm{div}(-K_1(|\nabla u|)\nabla u) = -b(u-v) \ \text{ in } S \times\Omega,\\
		-K_1(|\nabla u|)\nabla u \cdot {\bf n} = \varphi u^\lambda \ \text{ at } S\times\Gamma^R,\\
		-K_1(|\nabla u|)\nabla u\cdot {\bf n} = 0 \ \text{ at } S\times\Gamma^N,\\
		u(0,x) = u_0(x) \ \text{ for all } x \in \bar{\Omega}.
		\end{cases}
		\end{align}
		Multiplying both sides of the first equation in $(P_1)$ by $|u|^{\alpha+\delta-1}$ and integrating the result over $\Omega$, we obtain 
		\begin{align}
		&\int_{\Omega} \partial_t(u^\lambda)|u|^{\alpha + \delta -1}dx + \int_{\Omega}\mathrm{div} (-K_1(|\nabla u|)\nabla u)|u|^{\alpha + \delta - 1}dx = -\int_{\Omega}b(u-v)|u|^{\alpha + \delta -1}dx. \nonumber
		\end{align}
		Integrating by parts the last identity, and using the boundary conditions, it yields:
		\begin{align}
		\frac{\lambda}{\alpha}\frac{d}{dt}\int_{\Omega}|u|^\alpha dx + (\alpha - \lambda)\int_{\Omega} K_1(|\nabla u|)|\nabla u|^2u^{\alpha + \delta - 2}dx \nonumber\\+ \int_{\Gamma^R} |u|^{\alpha}\varphi d\gamma=-\int_{\Omega}b(u-v)|u|^{\alpha + \delta -1}dx. \nonumber
		\end{align}
		Using ($\text{A}_2$), we get the following estimate:
		\begin{align}\label{ineq_1}
		\frac{\lambda}{\alpha}\frac{d}{dt}\int_{\Omega}|u|^{\alpha}dx &+(\alpha - \lambda)\int_{\Omega} K_1(|\nabla u|)|\nabla u|^2|u|^{\alpha + \delta - 2}dx \leq \int_{\Omega}\hat{c}|u-v|^{\sigma}|u|^{\alpha + \delta - 1}dx \nonumber\\ &\leq \int_{\Omega}\hat{c}|u+v|^{\sigma}|u|^{\alpha + \delta -1}dx \nonumber\\
		&\leq \hat{c}\int_{\Omega}|u|^{\sigma}|u|^{\alpha + \delta -1}dx + \hat{c}\int_{\Omega}|v|^{\sigma}|u|^{\alpha + \delta -1}dx,
		\end{align}
		where we used the inequality $|u+v|^{\sigma} \leq |u|^{\sigma} + |v|^{\sigma}$ for $\sigma \in (0,1)$.
		
		By choosing $\sigma = 1- \delta$ and $\sigma \leq \frac{\alpha}{2}$ such that $\sigma \in (0,1)$ and $1<\alpha \leq 2$, the inequality \eqref{ineq_1} becomes
		\begin{align}\label{ineq_2}
		\frac{\lambda}{\alpha}\frac{d}{dt}\int_{\Omega}|u|^{\alpha}dx &+(\alpha - \lambda)\int_{\Omega} K_1(|\nabla u|)|\nabla u|^2|u|^{\alpha + \delta - 2}dx\nonumber\\
		&\leq \frac{\hat{c}}{2}\int_{\Omega}(1+|u|^{\alpha})dx + \frac{\hat{c}}{2}\int_{\Omega}(1+|v|^2)dx + \hat{c}\int_{\Omega}|u|^{\alpha}dx \nonumber\\
		&\leq \frac{\hat{c}}{2}|\Omega| + \frac{\hat{c}}{2}\int_{\Omega}|u|^{\alpha}dx + \hat{c}\int_{\Omega}|u|^\alpha dx + \frac{\hat{c}}{2}|\Omega| + \frac{\hat{c}}{2}\int_{\Omega}|v|^2dx \nonumber\\
		&\leq \hat{c}|\Omega| + \frac{3\hat{c}}{2}\|u\|^{\alpha}_{L^{\alpha}(\Omega)} + \frac{\hat{c}}{2}\|v\|_{L^2(\Omega)}^2.
		\end{align}
		\eqref{ineq_2} also leads to
		\begin{align}
		\frac{d}{dt}\int_{\Omega}|u|^{\alpha}dx &+ \int_{\Omega}|\nabla u|^{2-a}|u|^{\alpha+\delta-2}dx \leq \frac{d_3(\alpha-\lambda)+\hat{c}}{C_2}|\Omega| \nonumber\\&+ \left(\frac{3}{2C_2}\hat{c} + \frac{d_3(\alpha - \lambda)}{C_2}\right)\|u\|_{L^{\alpha}(\Omega)}^{\alpha} + \frac{\hat{c}}{2C_2}\|v\|_{L^2(\Omega)}^2,
		\end{align}
		where we use the property of $K_1$ as indicated in \eqref{k1-eqn1}.
		
		Now, we consider a second sub-problem which prefer to as $(P_2)$: For given $u \in L^\alpha(S; L^{\alpha}(\Omega))$, search for $v \in L^2(S; W^{1,2}(\Omega))$ such that
		\begin{align}
		\begin{cases}
		\partial_t v - K_2\Delta v =b(u-v) \ \text{ in } S\times\Omega, \\
		-K_2\nabla v \cdot {\bf n} = 0 \ \text{ at } S\times\partial\Omega, \\
		v(0,x) = v_0(x) \ \text{ for all } x \in \bar{\Omega}.
		\end{cases}
		\end{align}
		Multiplying the first equation of $(P_2)$ by $v$ and integrating the result over $\Omega$ lead to
		\begin{align}
		\int_{\Omega} \partial_tv v dx - \int_{\Omega}K_2 \Delta v vdx = \int_{\Omega}b(u-v)vdx.\nonumber
		\end{align}
		Integrating by parts this expression and using the corresponding boundary conditions ensure the identity:
		\begin{align}
		\frac{1}{2}\frac{d}{dt}\int_{\Omega}|v|^2dx + K_2\int_{\Omega}|\nabla v|^2dx = \int_{\Omega}b(u-v)vdx.
		\end{align}
		Then, by $(\text{A}_2)$, we have the following estimates
		\begin{align}\label{ineq_3}
		\frac{1}{2}\frac{d}{dt}\int_{\Omega}v^2dx + K_2\int_{\Omega}|\nabla v|^2dx &\leq \int_{\Omega}\hat{c}|u-v|^{\sigma}v dx \nonumber\\
		&\leq \hat{c}\int_{\Omega}|u|^{\sigma} v dx + \hat{c}\int_{\Omega}|v|^{\sigma}vdx \nonumber\\
		&\leq \frac{\hat{c}}{2}\int_{\Omega}|u|^{2\sigma}dx + \frac{\hat{c}}{2}\int_{\Omega}v^2dx + \hat{c}\int_{\Omega}(1+v^2)dx \nonumber\\
		&\leq \frac{\hat{c}}{2}\int_{\Omega}(1+|u|^{\alpha})dx + \frac{3\hat{c}}{2}\|v\|_{L^2(\Omega)}^2 + \hat{c}|\Omega| \nonumber\\
		&\leq \frac{3}{2}\hat{c}|\Omega| + \frac{\hat{c}}{2}\|u\|_{L^{\alpha}(\Omega)}^\alpha + \frac{3}{2}\hat{c}\|v\|_{L^2(\Omega)}^2.
		\end{align}
		Combining \eqref{ineq_2} and \eqref{ineq_3} together with neglecting the gradient terms from both these inequalities, we have
		\begin{align}\label{ine1}
		\frac{\lambda}{\alpha}\frac{d}{dt}\int_{\Omega}|u|^{\alpha}dx + \frac{1}{2}\frac{d}{dt}\int_{\Omega}|v|^2dx \leq \frac{5}{2}\hat{c}|\Omega| + 2\hat{c}\|u\|_{L^{\alpha}(\Omega)}^{\alpha} + 2 \hat{c}\|v\|_{L^2(\Omega)}^2.
		\end{align}
		Setting $\tilde{c}:= \min\left\{\frac{\lambda}{\alpha},\frac{1}{2}\right\}$, we rewrite \eqref{ine1} as
		\begin{align}\label{ineq_4}
		\frac{d}{dt}\int_{\Omega}|u|^{\alpha}dx + \frac{d}{dt}\int_{\Omega}|v|^2dx \leq \frac{5}{2}\underline{\tilde{c}}\hat{c}|\Omega| + 2\underline{\tilde{c}}\hat{c}\|u\|_{L^{\alpha}(\Omega)}^\alpha + 2\underline{\tilde{c}}\hat{c}\|v\|_{L^2(\Omega)}^2,
		\end{align}
		where $\underline{\tilde{c}}:=\frac{1}{\tilde{c}}$.
		
		For any $t\in S$, take $V(t):= 1 + \int_{\Omega}(|u|^{\alpha}+v^2)dx$ and $C_3:=\max\left\{\frac{5}{2}\underline{\tilde{c}}\hat{c}|\Omega|,2\underline{\tilde{c}}\hat{c}\right\}$. Then the inequality \eqref{ineq_4} becomes
		\begin{align}\label{bf-gronwall}
		\frac{d}{dt}V(t) \leq C_3V(t) \ \text{ for all } t \in S.
		\end{align}
		\eqref{bf-gronwall} leads to
		\begin{align}
		V(t) \leq V(0)\exp \left(\int_{0}^{t} C_3ds\right), \text{ and hence, we get } \nonumber
		\end{align}
		\begin{align}
		\int_{\Omega}|u|^{\alpha}dx + \int_{\Omega}v^2dx \leq e^{C_3t}\left(1+\|u_0\|_{L^\alpha(\Omega)}^{\alpha} + \|v_0\|_{L^2(\Omega)}^2\right).
		\end{align}
		Now we prove \eqref{ine_main2}. 
		Combining \eqref{ineq_2} and \eqref{ineq_3} yields
		\begin{align}\label{ineq_gron1}
		\frac{d}{dt}\int_{\Omega}|u|^{\alpha}dx + \frac{d}{dt}\int_{\Omega}|v|^2dx + (\alpha - \lambda)\int_{\Omega} K_1(|\nabla u|)|\nabla u|^2|u|^{\alpha + \delta - 2}dx \nonumber\\ +  K_2\int_{\Omega}|\nabla v|^2dx \leq \frac{5}{2}\hat{c}|\Omega|+2\hat{c}\|u\|_{L^\alpha(\Omega)}^\alpha+2\hat{c}\|v\|_{L^2(\Omega)}^2.
		\end{align}
		Set $C_4:= \min\left\{\alpha-\lambda,K_2\right\}$. Using \eqref{ine_main1} and integrating \eqref{ineq_gron1} over the time interval $S$, we are led to
		\begin{align}
		\int_0^T\int_{\Omega} K_1(|\nabla u|)|\nabla u|^2|u|^{\alpha + \delta - 2}dxdt+ \int_0^T\int_{\Omega}|\nabla v|^2dxdt \leq \frac{5T}{2C_4}\hat{c}|\Omega| + \frac{2T\hat{c}e^{C_3t}}{C_4} \nonumber\\ + \frac{2\hat{c}e^{C_3t}}{C_4} (\|u_0\|_{L^\alpha(S,L^\alpha(\Omega))}^{\alpha} + \|v_0\|_{L^2(S,L^2(\Omega))}^2).\nonumber
		\end{align} 
		Therefore, the inequality \eqref{ine_main2} holds. 
	\end{proof}
	\subsection{Gradient and time derivative estimates}
	We consider the following function $H: \Omega \longrightarrow \mathbb{R}$ given by
	\begin{align}
	H(\xi) = \int_{0}^{\xi^2}K_1(\sqrt{s})ds \ \text{ for } \xi \geq 0.
	\end{align}
	We have the following structural inequality between  $H(\xi)$ and $K_1(\xi)\xi^2$, i.e.
	\begin{align}\label{k1-ineq2}
	K_1(\xi)\xi^2 \leq H(\xi) \leq 2K_1(\xi)\xi^2 \text{ for all } \xi \geq 0.
	\end{align}
	By combining \eqref{k1-eqn1} and \eqref{k1-ineq2}, we deduce also that
	\begin{align}\label{H-ine1}
	d_3(\xi^{2-a}-1) \leq H(\xi) \leq 2d_2\xi^{2-a},
	\end{align}
	where $d_2, d_3$ and $a$ are defined as in $(\text{A}_1)$.
	\begin{proposition}
		Assume that ($A_1$)-($A_2$) hold. Let $\lambda \in (0,1]$, $\delta =1-\lambda$, $a=\frac{\alpha_N}{\alpha_N+1}$, $a > \delta$, $\alpha 
		\geq 2-\delta$, $\alpha\leq 2$ and $\sigma \leq \frac{\alpha}{2}$. Furthermore, suppose that $\nabla u_0 \in L^{\alpha}(\Omega) \cap L^{2-a}(\Omega)$, $u_0 \in L^{\alpha}(\Omega)$, $v_0 \in H^1(\Omega)$ and $\varphi \in L^{\infty}(\Gamma^R)$. Then, for any $t\in S$, the following estimates hold
		\begin{align}\label{ine-main3}
		\int_{\Omega}|\nabla u|^{2-a}dx + \int_{\Omega}|\nabla v|^2dx \leq C( \hat{c}, \lambda,a)\Bigg[\Lambda(0) + \int_{0}^{t}(1+\|u\|_{L^{\alpha}(\Omega)}^{\alpha})^{\beta}ds \nonumber\\+ \int_{0}^{t}\|v\|_{L^2(\Omega)}^2ds + \int_{0}^{t}\int_{\Gamma^R}|\varphi_t|^{\frac{\alpha}{\alpha-\lambda-1}}d\sigma ds\Bigg] + \int_{\Omega}|\nabla v_0|^2dx\nonumber\\
		+ \frac{\hat{c}^2}{C_2}|\Omega|t + \frac{\hat{c}^2}{2C_2} e^{C_3t}\left(1+\|u_0\|_{L^\alpha(S,L^{\alpha}(\Omega))}^{\alpha} + \|v_0\|_{L^2(S,L^2(\Omega))}^2\right) .
		\end{align}
		\begin{align}\label{ine-main4}
		\int_{\Omega}|(u^{\lambda})_t|^2dx + \int_{\Omega}|v_t|^2dx \leq C(\hat{c}, \lambda,a)\Bigg[1 + (1+\|u\|_{L^{\alpha}(\Omega)}^{\alpha})^{\beta} + \|v\|_{L^2(\Omega)}^2 \nonumber\\ + \int_{\Gamma^R}|\varphi_t|^{\frac{\alpha}{\alpha-\lambda-1}}d\sigma\Bigg] + \frac{\hat{c}^2}{C_2}|\Omega| + \frac{\hat{c}^2}{2C_2}e^{C_3t}\left(1+\|u_0\|_{L^{\alpha}(\Omega)}^{\alpha} + \|v_0\|_{L^2(\Omega)}^2\right),
		\end{align}
		where $C( \hat{c}, \lambda,a)>0$ is a constant and $$\Lambda(0):= \frac{\lambda+1}{2}\int_{\Omega}H(|\nabla u_0|)dx  + \int_{\Omega}|u_0|^{\alpha}dx.$$
	\end{proposition}\label{prop2}
	\begin{proof}
		We begin with studying the sub-problem ($P_1$) for a given choice of $v\in L^2(S; L^2(\Omega))$. Multiplying the first equation in $(P_1)$ by $u_t = \frac{1}{\lambda}(u^{\lambda})_tu^{1-\lambda}$ (note that $\frac{1}{\lambda}(u^{\lambda})_tu^{1-\lambda} = \frac{1}{\lambda}\lambda u^{\lambda-1}u_t u^{1-\lambda} = u_t$) and integrating the result over $\Omega$, we have
		\begin{align}\label{prop2-ine1}
		\int_{\Omega}\partial _t(u^{\lambda})u_tdx + \int_{\Omega}\mathrm{div}(-K_1(|\nabla u|)\nabla u)u_tdx = -\int_{\Omega}b(u-v)u_tdx. 
		\end{align}
		By integrating \eqref{prop2-ine1} by parts and using the property of the function $H$ as stated in \eqref{k1-ineq2}, yields:
		\begin{align}
		\frac{1}{\lambda}\int_{\Omega}u_t(u^{\lambda})_tdx + \frac{1}{2}\frac{d}{dt}\int_{\Omega}H(|\nabla u|)dx + \int_{\Gamma^R}\varphi u^{\lambda}u_td\sigma = -\int_{\Omega}b(u-v)u_tdx.
		\end{align}
		Using the assumption ($A_2$) together with the integration by parts the term $\int_{\Gamma^R}\varphi u^{\lambda}u_td\sigma$ and with applying afterwards the Cauchy-Schwarz's inequality, we get the upper bound
		\begin{align}\label{ineH_1}
		\frac{1}{\lambda}\int_{\Omega}|&u_t(u^{\lambda})_t|dx + \frac{1}{2}\frac{d}{dt}\int_{\Omega}H(|\nabla u|)dx \leq - \int_{\Gamma^R}\varphi u^{\lambda}u_td\sigma +\hat{c}\int_{\Omega}|u-v|^{\sigma}u_tdx \nonumber\\
		&\leq - \int_{\Gamma^R}\varphi |u|^{\lambda}|u_t|d\sigma+\hat{c}\int_{\Omega}|u+v|^{\sigma}|u_t|dx \nonumber\\
		&\leq -\frac{1}{\lambda + 1}\frac{d}{dt}\int_{\Gamma^R}|u|^{\lambda+1}\varphi d\sigma + \frac{1}{\lambda+1}\int_{\Gamma^R}|u|^{\lambda+1}\varphi_td\sigma +\nonumber\\&\hat{c}\int_{\Omega}|u|^{\sigma}|u_t|dx + \hat{c}\int_{\Omega}|v|^{\sigma}|u_t|dx \nonumber\\
		&\leq -\frac{1}{\lambda + 1}\frac{d}{dt}\int_{\Gamma^R}|u|^{\lambda+1}\varphi d\sigma + \frac{1}{\lambda+1}\int_{\Gamma^R}|u|^{\lambda+1}\varphi_td\sigma+ \frac{\hat{c}}{4\varepsilon}\int_{\Omega}|u|^{2\sigma}dx \nonumber\\&+ 2\varepsilon\hat{c}\int_{\Omega}|u_t|^2dx + \frac{\hat{c}}{4\varepsilon}\int_{\Omega}|v|^{2\sigma}dx\nonumber\\
		&\leq -\frac{1}{\lambda + 1}\frac{d}{dt}\int_{\Gamma^R}|u|^{\lambda+1}\varphi d\sigma + \frac{1}{\lambda+1}\int_{\Gamma^R}|u|^{\lambda+1}\varphi_td\sigma\nonumber\\&+\frac{\hat{c}}{2\varepsilon}|\Omega| + \frac{\hat{c}}{4\varepsilon}\int_{\Omega}|u|^{\alpha}dx + \frac{\hat{c}}{4\varepsilon}\int_{\Omega}|v|^2dx+ 2\varepsilon\hat{c}\int_{\Omega}|u_t|^2dx.
		\end{align}
		Multiplying the inequality \eqref{ineH_1} by $\lambda+1$ and applying Young's inequality to the term $\frac{1}{\lambda+1}\int_{\Gamma^R}|u|^{\lambda+1}\varphi_td\sigma$ yield the estimate
		\begin{align}\label{ineH_2}
		\frac{\lambda+1}{\lambda}\int_{\Omega}|u|^{1-\lambda}|(u^{\lambda})_t|^2dx + \frac{d}{dt}\left(\frac{\lambda+1}{2}\int_{\Omega}H(|\nabla u|)dx + \int_{\Gamma^R}|u|^{\lambda+1}\varphi d\sigma\right) \nonumber\\\leq \int_{\Gamma^R}|u|^{\alpha}d\sigma + \int_{\Gamma^R}|\varphi_t|^{\frac{\alpha}{\alpha-\lambda-1}}d\sigma 
		+\frac{\hat{c}}{2\varepsilon}|\Omega|(\lambda+1) + (\lambda+1)\frac{\hat{c}}{4\varepsilon}\int_{\Omega}|u|^{\alpha}dx \nonumber\\+ (\lambda+1)\frac{\hat{c}}{4\varepsilon}\int_{\Omega}|v|^2dx+ (\lambda+1)2\varepsilon\hat{c}\int_{\Omega}|u_t|^2dx.
		\end{align}
		Using the trace inequality \eqref{trace-theo}, we obtain
		\begin{align}\label{trace_1}
		\int_{\Gamma^R}|u|^{\alpha}d\sigma \leq 2\tilde{\varepsilon} \int_{\Omega}|\nabla u|^{2-a}|u|^{\alpha+\delta-2}dx + C\|u\|_{L^{\alpha}(\Omega)}^{\alpha} \nonumber\\+ C\tilde{\varepsilon}^{-\frac{1}{1-a}}\|u\|_{L^{\alpha}(\Omega)}^{\alpha+\mu_0} + C\tilde{\varepsilon}^{-\mu_2}\|u\|_{L^{\alpha}(\Omega)}^{\alpha+\mu_1}.
		\end{align}
		
		By \eqref{trace_1} and \eqref{ineH_2}, we obtain
		\begin{align}
		\frac{\lambda+1}{\lambda}\int_{\Omega}|u|^{1-\lambda}|(u^{\lambda})_t|^2dx + \frac{d}{dt}\left(\frac{\lambda+1}{2}\int_{\Omega}H(|\nabla u|)dx + \int_{\Gamma^R}|u|^{\lambda+1}\varphi d\sigma\right) \nonumber\\ \leq 2\tilde{\varepsilon} \int_{\Omega}|\nabla u|^{2-a}|u|^{\alpha+\delta-2}dx + C\|u\|_{L^{\alpha}(\Omega)}^{\alpha} + C\tilde{\varepsilon}^{-\frac{1}{1-a}}\|u\|_{L^{\alpha}(\Omega)}^{\alpha+\mu_0} + C\tilde{\varepsilon}^{-\mu_2}\|u\|_{L^{\alpha}(\Omega)}^{\alpha+\mu_1}\nonumber\\
		+ \int_{\Gamma^R}|\varphi_t|^{\frac{\alpha}{\alpha-\lambda-1}}d\sigma 
		+\frac{\hat{c}}{2\varepsilon}|\Omega|(\lambda+1) + (\lambda+1)\frac{\hat{c}}{4\varepsilon}\int_{\Omega}|u|^{\alpha}dx \nonumber\\+ (\lambda+1)\frac{\hat{c}}{4\varepsilon}\int_{\Omega}|v|^2dx+ (\lambda+1)2\varepsilon\hat{c}\int_{\Omega}|u_t|^2dx.
		\end{align}
		We denote $$\Lambda(t):= \frac{\lambda+1}{2}\int_{\Omega}H(|\nabla u|)dx  + \int_{\Omega}|u|^{\alpha}dx \text{ for all } t \in S.$$
		Note that
		$$\Lambda(0):= \frac{\lambda+1}{2}\int_{\Omega}H(|\nabla u_0|)dx  + \int_{\Omega}|u_0|^{\alpha}dx.$$
		Then, by using \eqref{main_0}, we have the following estimate
		\begin{align}
		\frac{\lambda+1}{\lambda}\int_{\Omega}|u|^{1-\lambda}|(u^{\lambda})_t|^2dx - 2\varepsilon\hat{c}\frac{\lambda+1}{\lambda}\int_{\Omega}|u|^{2(1-\lambda)}|(u^{\lambda})_t|^2dx  + \frac{d}{dt}\Lambda(t) \nonumber\\+(1-2\tilde{\varepsilon})\int_{\Omega}|u|^{\alpha+\delta-2}|\nabla u|^{2-a}dx \leq 
		\left(\frac{d_3(\alpha-\lambda) + \hat{c}}{C_2}+\frac{(\lambda+1)\hat{c}}{2\varepsilon}\right)|\Omega| \nonumber\\+ \left(\frac{3}{2C_2}\hat{c}+\frac{d_3(\alpha-\lambda)}{C_2} + C + \frac{(\lambda+1)\hat{c}}{4\varepsilon}\right)\|u\|_{L^{\alpha}(\Omega)}^{\alpha}+ \left(\frac{\hat{c}}{2C_2} + \frac{\hat{c}(\lambda+1)}{4\varepsilon}\right)\|v\|_{L^2(\Omega)}^2 \nonumber\\+ C\tilde{\varepsilon}^{-\frac{1}{1-a}}\|u\|_{L^{\alpha}(\Omega)}^{\alpha+\mu_0} + C\tilde{\varepsilon}^{-\mu_2}\|u\|_{L^{\alpha}(\Omega)}^{\alpha+\mu_1} + \int_{\Gamma^R}|\varphi_t|^{\frac{\alpha}{\alpha-\lambda-1}}d\sigma. \nonumber
		\end{align}
		Based on ($A_3$), after choosing $\tilde{\varepsilon}=\frac{1}{4}$ and $\varepsilon=\frac{1}{4\hat{c}}$, we obtain
		\begin{align}\label{ine-1}
		\int_{\Omega}|(u^{\lambda})_t|^2dx + \frac{d}{dt}\Lambda(t) &\leq C(\hat{c},\lambda, a) \Big[1+ \|u\|_{L^{\alpha}(\Omega)}^{\alpha} +\|v\|_{L^2(\Omega)}^2 \nonumber\\& +\|u\|_{L^{\alpha}(\Omega)}^{\alpha+\mu_0} + \|u\|_{L^{\alpha}(\Omega)}^{\alpha+\mu_1} +  \int_{\Gamma^R}|\varphi_t|^{\frac{\alpha}{\alpha-\lambda-1}}d\sigma\Big]. 
		\end{align}
		We denote $\beta:=\alpha+\mu_1$, and $\beta$ is the maximum allowed power of $\|u\|_{L^{\alpha}(\Omega)}$ when considering the right hand side of \eqref{ine-1}. Now, using the inequality \eqref{pre-ine-5} leads to
		\begin{align}\label{ine-2}
		\int_{\Omega}|(u^{\lambda})_t|^2dx + \frac{d}{dt}\Lambda(t) \leq C( \hat{c}, \lambda,a)\Bigg[1 + (1+\|u\|_{L^{\alpha}(\Omega)}^{\alpha})^{\beta} \nonumber\\+ \|v\|_{L^2(\Omega)}^2 + \int_{\Gamma^R}|\varphi_t|^{\frac{\alpha}{\alpha-\lambda-1}}d\sigma\Bigg].
		\end{align} 
		Integrating \eqref{ine-2} over the time interval $(0,t)$, we are led to
		\begin{align}\label{time-1}
		\int_{0}^{t}\int_{\Omega}|(u^{\lambda})_t|^2dx + \frac{\lambda+1}{2}\int_{\Omega}H(|\nabla u|)dx + \int_{\Omega}|u|^{\alpha}dx \leq C(\hat{c}, \lambda,a)\Bigg[\Lambda(0) \nonumber\\ + \int_{0}^{t}(1+\|u\|_{L^{\alpha}(\Omega)}^{\alpha})^{\beta}ds + \int_{0}^{t}\|v\|_{L^2(\Omega)}^2ds + \int_{0}^{t}\int_{\Gamma^R}|\varphi_t|^{\frac{\alpha}{\alpha-\lambda-1}}d\sigma ds\Bigg].
		\end{align} 
		This fact also implies 
		\begin{align}
		\int_{0}^{t}\int_{\Omega}|(u^{\lambda})_t|^2dx + \int_{\Omega}H(|\nabla u|)dx + \int_{\Omega}|u|^{\alpha}dx \leq C( \hat{c}, \lambda,a)\Bigg[\Lambda(0) \nonumber\\ + \int_{0}^{t}(1+\|u\|_{L^{\alpha}(\Omega)}^{\alpha})^{\beta}ds + \int_{0}^{t}\|v\|_{L^2(\Omega)}^2ds + \int_{0}^{t}\int_{\Gamma^R}|\varphi_t|^{\frac{\alpha}{\alpha-\lambda-1}}d\sigma ds\Bigg].
		\end{align}
		Employing \eqref{H-ine1} yields
		\begin{align}\label{gra_u}
		\int_{0}^{t}\int_{\Omega}|(u^{\lambda})_t|^2dx + \int_{\Omega}|\nabla u|^{2-a}dx \leq C( \hat{c}, \lambda,a)\Bigg[\Lambda(0) + \int_{0}^{t}(1+\|u\|_{L^{\alpha}(\Omega)}^{\alpha})^{\beta}ds \nonumber\\+ \int_{0}^{t}\|v\|_{L^2(\Omega)}^2ds + \int_{0}^{t}\int_{\Gamma^R}|\varphi_t|^{\frac{\alpha}{\alpha-\lambda-1}}d\sigma ds\Bigg].
		\end{align}
		
		Now, we consider the sub-problem ($P_2$): Take a fixed $u\in L^2(S; L^{\alpha}(\Omega))$. Multiplying the first equation of $(P_2)$ by $v_t$ and then integrating the result over $\Omega$, we have
		\begin{align}
		\int_{\Omega}|v_t|^2 dx + \frac{K_2}{2}\frac{d}{dt}\int_{\Omega}|\nabla v|^2 dx = \int_{\Omega}b(u-v)v_tdx.
		\end{align}
		By ($A_2$) together with \eqref{pre-ine-2} and \eqref{pre-ine-5}, it results
		\begin{align}\label{ine_gr2_1}
		\int_{\Omega}|v_t|^2 dx + \frac{K_2}{2}\frac{d}{dt}\int_{\Omega}|\nabla v|^2 dx &\leq \hat{c}\int_{\Omega}|u-v|^{\sigma}|v_t|dx \nonumber\\
		&\leq \hat{c}\int_{\Omega}|u+v|^{\sigma}|v_t|dx \nonumber\\
		&\leq \hat{c}\int_{\Omega}|u|^{\sigma}|v_t| dx + \hat{c}\int_{\Omega}|v|^{\sigma}|v_t|dx \nonumber\\
		&\leq \frac{\hat{c}}{4\varepsilon}\int_{\Omega}|u|^{2\sigma}dx  + \frac{\hat{c}}{4\varepsilon}\int_{\Omega}|v|^{2\sigma}dx + 2\hat{c}\varepsilon\int_{\Omega}|v_t|^2dx \nonumber\\
		&\leq \frac{\hat{c}}{2\varepsilon}|\Omega| + \frac{\hat{c}}{4\varepsilon}\int_{\Omega}|u|^{\alpha}dx \nonumber\\
		&+ \frac{\hat{c}}{4\varepsilon}\int_{\Omega}|v|^2dx + 2\hat{c}\varepsilon\int_{\Omega}|v_t|^2dx.
		\end{align}
		If we choose $\varepsilon = \frac{1}{4\hat{c}}$ such that $1-2\varepsilon\hat{c}>0$, then \eqref{ine_gr2_1} becomes
		\begin{align}
		\frac{1}{2}\int_{\Omega}|v_t|^2dx + \frac{K_2}{2}\frac{d}{dt}\int_{\Omega}|\nabla v|^2dx \leq 2\hat{c}^2|\Omega| + \hat{c}^2\int_{\Omega}|u|^\alpha dx + \hat{c}^2\int_{\Omega}|v|^2dx.
		\end{align}
		Putting $C_5:=\min\left\{\frac{1}{2},\frac{K_2}{2}\right\}$, we get the following estimate
		\begin{align}\label{ine_gr2_2}
		\int_{\Omega}|v_t|^2dx + \frac{d}{dt}\int_{\Omega}|\nabla v|^2dx \leq \frac{\hat{c}^2}{C_2}|\Omega|+\frac{\hat{c}^2}{C_2}\int_{\Omega}|u|^{\alpha}dx + \frac{\hat{c}^2}{C_2}\int_{\Omega}|v|^2dx.  
		\end{align}
		Applying \eqref{ine_main1} to the right hand side of \eqref{ine_gr2_2}, we obtain
		\begin{align}\label{ine_gr2_3}
		\int_{\Omega}|v_t|^2dx + \frac{d}{dt}\int_{\Omega}|\nabla v|^2dx \leq \frac{\hat{c}^2}{C_2}|\Omega| + \frac{\hat{c}^2}{2C_2}e^{C_1t}\left(1+\|u_0\|_{L^{\alpha}(\Omega)}^{\alpha} + \|v_0\|_{L^2(\Omega)}^2\right).
		\end{align}
		The inequality \eqref{ine_gr2_3} implies
		\begin{align}\label{time-2}
		\int_{\Omega}|v_t|^2dx \leq \frac{\hat{c}^2}{C_2}|\Omega| + \frac{\hat{c}^2}{2C_2}e^{C_3t}\left(1+\|u_0\|_{L^{\alpha}(\Omega)}^{\alpha} + \|v_0\|_{L^2(\Omega)}^2\right).
		\end{align}
		On the other hand, \eqref{ine_gr2_3} also leads to
		\begin{align}\label{ine_gr_4}
		\frac{d}{dt}\int_{\Omega}|\nabla v|^2dx \leq \frac{\hat{c}^2}{C_2}|\Omega| + \frac{\hat{c}^2}{2C_2}e^{C_3t}\left(1+\|u_0\|_{L^{\alpha}(\Omega)}^{\alpha} + \|v_0\|_{L^2(\Omega)}^2\right).
		\end{align}
		Integrating \eqref{ine_gr_4} over the time interval $(0,t)$, we obtain the upper bound
		\begin{align}\label{gra_v}
		\int_{\Omega}|\nabla v(t)|^2dx &\leq \int_{\Omega}|\nabla v_0|^2dx + \frac{\hat{c}^2}{C_2}|\Omega|t \nonumber\\&+ \frac{\hat{c}^2}{2C_2} e^{C_3t}\left(1+\|u_0\|_{L^\alpha(S; L^{\alpha}(\Omega))}^{\alpha} + \|v_0\|_{L^2(S; L^2(\Omega))}^2\right).
		\end{align}
		Combining \eqref{gra_u} and \eqref{gra_v}, we obtain
		\begin{align}
		\int_{\Omega}|\nabla u|^{2-a}dx + \int_{\Omega}|\nabla v|^2dx &\leq C( \hat{c}, \lambda,a)\Bigg[\Lambda(0) + \int_{0}^{t}(1+\|u\|_{L^{\alpha}(\Omega)}^{\alpha})^{\beta}ds \nonumber\\&+ \int_{0}^{t}\|v\|_{L^2(\Omega)}^2ds + \int_{0}^{t}\int_{\Gamma^R}|\varphi_t|^{\frac{\alpha}{\alpha-\lambda-1}}d\sigma ds\Bigg] \nonumber\\
		&+ \int_{\Omega}|\nabla v_0|^2dx+ \frac{\hat{c}^2}{C_2}|\Omega|t \nonumber\\
		&+ \frac{\hat{c}^2}{2C_2} e^{C_3t}\left(1+\|u_0\|_{L^\alpha(S; L^{\alpha}(\Omega))}^{\alpha} + \|v_0\|_{L^2(S; L^2(\Omega))}^2\right) .
		\end{align}
		Combining \eqref{time-1} and \eqref{time-2}, we obtain
		\begin{align}
		\int_{\Omega}|(u^{\lambda})_t|^2dx + \int_{\Omega}|v_t|^2dx \leq C\Big[1 + (1+\|u\|_{L^{\alpha}(\Omega)}^{\alpha})^{\beta} + \|v\|_{L^2(\Omega)}^2 + \int_{\Gamma^R}|\varphi_t|^{\frac{\alpha}{\alpha-\lambda-1}}d\sigma\Big] \nonumber\\+ \frac{\hat{c}^2}{C_2}|\Omega| + \frac{\hat{c}^2}{2C_2}e^{C_3t}\left(1+\|u_0\|_{L^{\alpha}(\Omega)}^{\alpha} + \|v_0\|_{L^2(\Omega)}^2\right).
		\end{align}
		This completes the proof of the theorem.
	\end{proof}
	\section{Proof of Theorem \ref{existence-theo}}\label{existence}
	\begin{proof}
		By using Schauder's fixed point argument (see e.g. Theorem \ref{Schauder-theo}), we show that there exist a pair $(u,v)$ of weak solutions to problem \eqref{main_eq} in the sense of Definition \ref{dn-weak}. First of all, let us define the operators:
		\begin{align*}
		\mathcal{T}_1: L^2(S;L^{\alpha}(\Omega)\cap W^{1, 2-a}(\Omega)) \longrightarrow L^2(S;L^2(\Omega))
		\end{align*} 
		by $\mathcal{T}_1(u)=v$ and
		\begin{align*}
		\mathcal{T}_2: L^2(S; L^2(\Omega)) \longrightarrow L^\alpha(S;L^\alpha(\Omega)\cap W^{1, 2-a}(\Omega))
		\end{align*}
		by $\mathcal{T}_2(v)=w$. Then, consider the operator $\mathcal{T}: L^\alpha(S;L^{\alpha}(\Omega)) \longrightarrow L^\alpha(S;L^\alpha(\Omega)\cap W^{1, 2-a}(\Omega))$ defined by
		\begin{align}
		\mathcal{T}(w) = \mathcal{T}_2(\mathcal{T}_1(u))=w.
		\end{align}
		Indeed, the estimates reported in Proposition \ref{prop1} and Proposition \ref{prop2} imply that the operators $\mathcal{T}_1$ and $\mathcal{T}_2$ are well-defined. Hence, the operator $\mathcal{T}$ is also well-defined. 
		
		In order to show the existence of solution to the problem \eqref{main_eq}, we wish to show that $\mathcal{T}$ admits a fixed point. Then, using Schauder's fixed point Theorem \ref{Schauder-theo}, we shall prove that there exits a set $B$ such that
		\begin{itemize}
			\item[(1)] $\mathcal{T}: B \rightarrow B$ is a compact operator;
			\item[(2)] $B$ is convex, closed, bounded set such that $\mathcal{T}(B) \subset B$.
		\end{itemize}
		In particular, to obtain the compactness of $\mathcal{T}=\mathcal{T}_2\circ \mathcal{T}_1$, it is sufficient to demonstrate that $\mathcal{T}_1$ is compact and that $\mathcal{T}_2$ is continuous.
		
		Recall that we have
		\begin{align*}
		\mathcal{T}_1: L^2(S;L^{\alpha}(\Omega)\cap W^{1, 2-a}(\Omega)) \longrightarrow L^2(S;L^2(\Omega)).
		\end{align*}
		We assume in Proposition \ref{prop1} that for given $u\in L^{\alpha}(S;L^{\alpha}(\Omega))$, we obtain $\mathcal{T}_1(u) = v \in L^2(S; W^{1,2}(\Omega))$ with $v_t\in L^2(S;L^2(\Omega))$. Hence,
		$$\mathcal{T}_1(L^2(S;L^{\alpha}(\Omega)\cap L^2(\Omega))) \subset V,$$
		where $$V=\{\varphi: \varphi\in L^2(S; L^{\alpha}(\Omega)\cap W^{1,2-a}(\Omega)), \partial_t \varphi \in L^2(S; L^2(\Omega))\}.$$
		By using Rellich-Kondrachov's Theorem (of Theorem \ref{rellich-theo}), we obtain
		$$L^{\alpha}(\Omega)\cap W^{1,2-a}(\Omega) \hookrightarrow\hookrightarrow L^2(\Omega).$$
		Applying Theorem \ref{Aubin-theo} gives $V\hookrightarrow\hookrightarrow L^2(S; L^2(\Omega))$. Thus, for any bounded set $M \subset L^2(S; L^2(\Omega))$, then we have $\mathcal{T}_1(M) \subset V$. Since $V$ is compactly embedded in $L^2(S; L^2(\Omega))$, then we have $\mathcal{T}_1(M)$ is precompact in $L^2(S; L^2(\Omega))$. Therefore, $\mathcal{T}_1$ is a  compact operator. 
		
		Now, we prove that $\mathcal{T}_2$ is sequentially continuous. We proceed in a similar manner as in   \cite{Arendt2010}. We recall first that
		$$\mathcal{T}_2: L^2(S; L^2(\Omega)) \longrightarrow L^\alpha(S; L^\alpha(\Omega) \cap W^{1, 2-a}(\Omega)).$$
		Let $v_n \longrightarrow v$ in $L^2(S; L^2(\Omega))$ as $n \longrightarrow \infty$ and with $u_n=\mathcal{T}(v_n)$ and $u=\mathcal{T}(v)$. We show that $u_n \longrightarrow u$ in $L^{\alpha}(S; L^\alpha(\Omega))$ as $n \longrightarrow \infty$.
		
		We denote $$E:=\{\varphi: \varphi \in L^\alpha(S; L^\alpha(\Omega)) \cap L^{2-a}(S; W^{1,2-a}(\Omega))\}.$$
		Since $L^\alpha(\Omega)$ and $W^{1,2-a}(\Omega)$ are reflexive Banach spaces, then also the Bochner spaces $L^\alpha(S; L^\alpha(\Omega))$ 
		and $L^{2-a}(S; W^{1,2-a}(\Omega))$ are reflexive Banach spaces. Thus, $E$ is a reflexive Banach space. We know that from a bounded sequence in a reflexive Banach space $E$, one can extract a subsequence that converges weakly in $E$ in the weak topology (cf. Theorem $3.18$, \cite{Brezis2010}). Indeed, since we have $(u_n)$ is bounded in $E$ and $v_n \longrightarrow v$ as $n\rightarrow \infty$ in $L^2(S; L^2(\Omega))$, then we have $u_n \rightharpoonup u$ as $n\rightarrow \infty$ in $E$.
		
		By the estimates \eqref{ine_main1}, \eqref{ine-main3} and \eqref{ine-main4}, we can extract two subsequences $(u_{n_k})$ and $(v_{n_k})$, still labeled by $n$ instead of $n_k$ for simplicity, such that as $n \to \infty$ it holds:
		\begin{align}
		\nabla v_n &\rightharpoonup \nabla v \text{ in } L^2(S; L^2(\Omega)), \nonumber\\
		(v_n)_t &\rightharpoonup v_t \text{ in } L^2(S; L^2(\Omega)), \nonumber\\
		u_n &\rightharpoonup u \text{ in } E,\nonumber\\
		(u_n^\lambda)_t &\rightharpoonup (u^\lambda)_t \text{ in } L^2(S; L^2(\Omega)),\nonumber\\
		\nabla u_n &\rightharpoonup \nabla u \text{ in } L^2(S; W^{1,2-a}(\Omega)). \nonumber
		\end{align} 
		Then, by using ($A_1$) and ($A_2$) and the fact that $b(u_n-v_n)$ and $K_1(|\nabla u_n|)|\nabla u_n|^2$ are bounded in $E$, it leads to
		\begin{align}
		b(u_n-v_n) &\rightharpoonup b(u-v) \text{ in } E,  \nonumber\\
		K_1(|\nabla u_n|)|\nabla u_n|^2 &\rightharpoonup K_1(|\nabla u|)|\nabla u|^2 \text{ in } E. \nonumber
		\end{align}
		We re-write \eqref{main_eq} formulated for the sequences $(u_n) \in L^\alpha(S; L^{\alpha}(\Omega))$ and $(v_n) \in L^2(S; L^2(\Omega))$
		\begin{align}\label{main_eq2}
		\begin{cases}
		\partial_t (u_n^\lambda) + \mathrm{div}(-K_1(|\nabla u_n|)\nabla u_n) = -b(u_n-v_n) \ \text{ in }  S\times \Omega,\\
		\partial_t v_n - K_2\Delta v_n =b(u_n-v_n) \ \text{ in } S \times\Omega, \\
		-K_1(|\nabla u_n|)\nabla u_n \cdot {\bf n} = \varphi u_n^\lambda \ \text{ at } S\times \Gamma^R,\\
		-K_1(|\nabla u_n|)\nabla u_n\cdot {\bf n} = 0 \ \text{ at } S\times \Gamma^N,\\
		-K_2\nabla v_n \cdot {\bf n} = 0 \ \text{ at } S \times\partial\Omega, \\
		u_n(t=0,x) = u_{0_n}(x), \  x\in\bar{\Omega}, \\
		v_n(t=0,x) = v_{0_n}(x), \ x\in  \bar{\Omega}.
		\end{cases}
		\end{align}
		Clearly, if $n \longrightarrow \infty$ in the weak form of \eqref{main_eq2} recovers the weak form of \eqref{main_eq}. Essentially,
		we have shown that $u_n \rightharpoonup u$ in $E$. Moreover, the embedding $E \hookrightarrow L^\alpha(S; L^\alpha(\Omega))$ is compact, this implies that $u_n \longrightarrow u$ in $L^\alpha(S; L^\alpha(\Omega))$. Therefore, $\mathcal{T}_2$ is continuous.
		
		Let us fix $K>0$ to be specified later and we denote by $B_K$ the collection of functions $u \in L^\alpha(S; L^\alpha(\Omega) \cap W^{1,2-a}(\Omega))$ such that
		\begin{align*}
		\max\{\|u\|_{L^\alpha(S; L^\alpha(\Omega))}, \|\nabla u\|_{L^{2-a}(S; L^{2-a}(\Omega))}\} \leq K.
		\end{align*}
		For each choice of $K$, the set $$B_K\subset L^\alpha(S; L^\alpha(\Omega) \cap W^{1,2-a}(\Omega))$$
		is convex, closed, and bounded. We aim to show that we may select a $K>0$ such that $$\mathcal{T}(B_K) \subset B_K.$$
		Indeed, by using the estimates \eqref{ine_main1}, \eqref{ine-main3} as well as the fact that $$L^2(S; L^\alpha(\Omega) \cap W^{1,2-a}(\Omega)) \subset \left(L^\alpha(S;L^\alpha(\Omega))\cap L^{2-a}(S;W^{1,2-a}(\Omega))\right)$$ together with knowing that $\mathcal{T}_1(B_K)$ is bounded subset of $L^2(S; L^2(\Omega))$ and that $\mathcal{T}_2(\mathcal{T}_1(u))$ is bounded subset of $L^{\alpha}(S; L^\alpha(\Omega)\cap W^{1,2-a}(\Omega))$, we have
		\begin{align*}
		\max\{\|u\|_{L^\alpha(S; L^\alpha(\Omega))}, \|\nabla u\|_{L^{2-a}(S; L^{2-a}(\Omega))}\} \leq K.
		\end{align*}
		Here $K>0$ is chosen such that
		\begin{align}
		K&:=\max\Big\{e^{C_3T}\left(1+\|u_0\|_{L^\alpha(0,T; L^\alpha(\Omega))}^{\alpha} + \|v_0\|_{L^2(0,T;L^2(\Omega))}^2\right),
		\nonumber\\ 
		&C( \hat{c}, \lambda,a)\Big[\Lambda(0) + \int_{0}^{t}(1+\|u\|_{L^{\alpha}(\Omega)}^{\alpha})^{\beta}ds 
		\nonumber\\
		&+ \int_{0}^{t}\|v\|_{L^2(\Omega)}^2ds + \int_{0}^{t}\int_{\Gamma^R}|\varphi_t|^{\frac{\alpha}{\alpha-\lambda-1}}d\sigma ds\Big] + \int_{\Omega}|\nabla v_0|^2dx \nonumber\\
		&+ \frac{\hat{c}^2}{C_2}|\Omega|t + \frac{\hat{c}^2}{2C_2} e^{C_3t}\left(1+\|u_0\|_{L^\alpha(0,T; L^{\alpha}(\Omega))}^{\alpha} + \|v_0\|_{L^2(0,T; L^2(\Omega))}^2\right) \Big\}.
		\end{align}
		Hence, $\mathcal{T}(B_K) \subset B_K$.
		
		We have shown that $\mathcal{T}: B_k \longrightarrow B_k$ is a compact operator with $B_K$ a convex, closed, bounded set and also that $\mathcal{T}(B_K) \subset B_K$. 
		Then, by Theorem \ref{Schauder-theo} there exists at least a pair $$(u,v) \in \left(L^\alpha(S;L^\alpha(\Omega))\cap L^{2-a}(S;W^{1,2-a}(\Omega))\right) \times L^2(S;W^{1,2}(\Omega))$$ satisfying the problem \eqref{main_eq} in the sense of Definition \ref{dn-weak}.
	\end{proof}

		\begin{center}
		\bibliographystyle{alpha}
		\bibliography{mybib}

\begin{thebibliography}{CMRT19}

\bibitem[ABHI09]{Aulisa2009}
E.~Aulisa, L.~Bloshanskaya, L.~Hoang, and A.~Ibragimov.
\newblock Analysis of generalized {F}orchheimer flows of compressible fluids in
  porous media.
\newblock {\em Journal of Mathematical Physics}, 50:103102, 2009.

\bibitem[AC10]{Arendt2010}
W.~Arendt and R.~Chill.
\newblock Global existence for quasilinear diffusion equations in isotropic
  nondivergence form.
\newblock {\em Ann. Scuola Norm. Sup. Pisa Cl. Sci.}, IX:523--539, 2010.

\bibitem[Aub63]{Aubin1963}
J.~P. Aubin.
\newblock Un th\'{e}oreme de compacit\'{e}.
\newblock {\em C R Acad Sci Paris}, 256:5042--5044, 1963.

\bibitem[Bre11]{Brezis2010}
H.~Brezis.
\newblock {\em Functional Analysis, Sobolev Spaces and Partial Differential
  Equations}.
\newblock Springer, 2011.

\bibitem[BZK60]{Barenblatt1960}
G.~I. Barenblatt, Iu.~P. Zheltov, and I.~N. Kochina.
\newblock Basic concepts in the theory of seepage of homogeneous liquids in
  fissured rocks.
\newblock {\em PMM}, 24:852--864, 1960.

\bibitem[CCMT19]{Cirillo2019}
E.~N.~M. Cirillo, M.~Colangeli, A.~Muntean, and T.~K.~T. Thieu.
\newblock A lattice model for active-passive pedestrian dynamics: a quest for
  drafting effects.
\newblock accepted to Mathematical Biosciences and Engineering
  (arXiv:1907.08621), 2019.

\bibitem[CHK16]{Celik2016}
E.~Celik, L.~Hoang, and T.~Kieu.
\newblock Generalized {F}orchheimer flows of isentropic gases.
\newblock {\em J. Math. Fluid Mech.}, 20:83--115, 2016.

\bibitem[CMRT19]{Colangeli2019}
M.~Colangeli, A.~Muntean, O.~Richardson, and T.~K.~T. Thieu.
\newblock {\em Modelling interactions between active and passive agents moving
  through heterogeneous environments}, volume 1: Theory, Models and Safety
  Problems,.
\newblock in G. Libelli, N. Bellomo (Eds), Crowd Dynamics, Modeling and
  Simulation in Science, Engineering and Technology, Boston, Birkhauser,
  Springer, 2019.

\bibitem[Eva98]{Evans1998}
L.~C. Evans.
\newblock {\em Partial Differential Equations}.
\newblock American Mathematical Society, 1998.

\bibitem[HI11]{Hoang2011}
L.~Hoang and A.~Ibragimov.
\newblock Structural stability of generalized {F}orchheimer equations for
  compressible fluids in porous media.
\newblock {\em Nonlinearity}, 24:1--41, 2011.

\bibitem[LLPW11]{Lions2011}
J.~L. Lions, D.~Lukkassen, L.~E. Persson, and P.~Wall.
\newblock Reiterated homogenization of nonlinear monotone operators.
\newblock {\em Chinese Annals of Mathematics}, 22:1--12, 2011.

\bibitem[LMR18]{Lind2018}
M.~Lind, A.~Muntean, and O.~M. Richardson.
\newblock Well-posedness and inverse {R}obin estimate for a multiscale
  elliptic/parabolic system.
\newblock {\em Applicable Analysis}, 97:89--106, 2018.

\bibitem[LVAS68]{Ladyzenskaja1968}
O.~A. Ladyzenskaja and N.~N.~Uralceva V.~A.~Solonnikov.
\newblock {\em Linear and Quasilinear Equations of Parabolic Type}, volume~23.
\newblock American Mathematical Society, 1968.

\bibitem[MNR10]{Muntean2010}
A.~Muntean and M.~Neuss-Radu.
\newblock A multiscale {G}alerkin approach for a class of nonlinear coupled
  reaction-diffusion systems in a complex media.
\newblock {\em Journal of Mathematical Analysis and Applications},
  371:705--718, 2010.

\bibitem[RJM19]{Richardson2019}
O.~Richardson, A.~Jalba, and A.~Muntean.
\newblock The effect of environment knowledge in evacuation scenarios involving
  fire and smoke – a multiscale modelling and simulation approach.
\newblock {\em Fire Technology}, 55:415--436, 2019.

\bibitem[Zei86]{Zeidler1986}
E.~Zeidler.
\newblock {\em Nonlinear {F}unctional {A}nalysis and its {A}pplications},
  volume~1.
\newblock 9th ed. Springer-Verlag, 1986.

\end{thebibliography}
	\end{center}

\end{document}